\documentclass[12pt]{article}

\usepackage[margin=1.5in]{geometry}  
\usepackage{graphicx}              
\usepackage{amsmath}               
\usepackage{amsfonts}              
\usepackage{amsthm}                
\usepackage{tikz}
\usetikzlibrary{positioning}
\usepackage{tkz-graph}
\usepackage{float}
\GraphInit[vstyle = Shade]

\newtheorem{thm}{Theorem}[section]
\newtheorem{lem}[thm]{Lemma}

\newtheorem{definition}{Definition}

\begin{document}


\title{\textbf{On Counting Constructions and Isomorphism Classes of $I$-Graphs}}

\author{Harrison Bohl \\ 
School of Mathematics and Physics \\
University of Queensland \\ 
\texttt{bohlharrison357@gmail.com} \\ \ \\
Adrian W. Dudek \\ 
School of Mathematics and Physics \\
University of Queensland \\ 
\texttt{a.dudek@uq.edu.au}}

\date{}

\maketitle

\begin{abstract}
We prove a collection of asymptotic density results for several interesting classes of the $I$-graphs. Specifically, we quantify precisely the proportion of $I$-graphs that are generalised Petersen graphs as well as those that are connected. Our results rely on the estimation of sums over tuples satisfying various coprimality conditions along with other techniques from analytic number theory.
\end{abstract}

\section{Introduction and Main Results}

The purpose of this paper is to prove some results regarding the $I$-graphs, a further generalisation of the generalised Petersen graphs (see \cite{Boben} for an introduction). Specifically, our main result is to quantify the extent of this generalisation, that is, we determine the density of the $I$-graphs that are isomorphic to generalised Petersen graphs. 

We first dispense with terminology; if one lets $n, k \in \mathbb{N}$ such that $n\geq 3$ and $k \leq n/2$, then one can define the generalised Petersen graph (or GPG) $P(n,k)$ to be the graph with vertex set
$$V=\{a_i, b_i : 0 \leq i \leq n-1\}$$
and edge set
$$E=\{a_i a_{i+1}, a_i b_i, b_i b_{i+k} : 0 \leq i \leq n-1 \}$$
with the subscript arithmetic being performed modulo $n$. 

\begin{figure}[H]
\centering
\includegraphics[width=0.45\linewidth,angle=270]{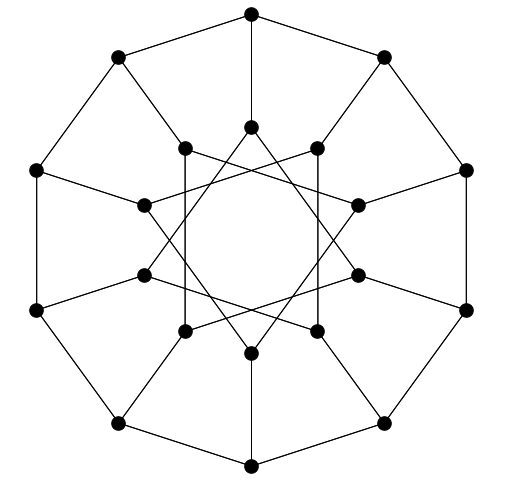}
\caption{$GPG(10,3)$.}
\label{fig:GPGex}
\end{figure}

Moreover, throughout this paper, $\zeta(s)$ refers to the Riemann zeta-function, $\phi(n)$ is the Euler totient function, $\omega(n)$ denotes the number of distinct prime divisors of $n$, $\tau(n)$ denotes the number of divisors of $n$, and $\mu(n)$ denotes the M\"obius function evaluated at $n$. One can see the textbook of Bateman--Diamond \cite{BatemanDiamond} for more information on these and other functions central to analytic number theory.

To further generalise the GPGs, let $n, j,k \in \mathbb{N}$ such that $n \geq 3$, $k \leq n/2$ and $1 \leq j \leq k$, and define the $I$-graph $I(n,j,k)$ to be the graph with vertex set
$$V=\{a_i, b_i : 0 \leq i \leq n-1\}$$
and edge set
$$E=\{a_i a_{i+j}, a_i b_i, b_i b_{i+k} : 0 \leq i \leq n-1 \}.$$
Clearly, we have that $I(n,1,k) = P(n,k)$ for all $n$ and $k \leq n/2$. 

The main purpose of this paper is to prove some results on the number of $I$-graphs that are GPGs as well as on the number of $I$-graphs that are connected. We begin with the statement of a result which says that, in the asymptotic sense, approximately 89.32\% of permissible choices for $(n,j,k)$ arise in an $I$-graph that is also a generalised Petersen graph.

\begin{thm} \label{igraphsextendgpgs}
Let $A(N)$ count the number of tuples $(n, j, k)$ with $3 \leq n \leq N$, $k \leq n/2$ and $1 \leq j \leq k$. Let $B(N)$ count the number of these tuples such that $I(n, j, k)$ is a generalised Petersen graph. Then 
$$\lim_{N \rightarrow \infty} \frac{B(N)}{A(N)} = \frac{12}{\pi^2} - C = 0.8932\ldots$$
where $C = \prod_p (1-2/p^2) = 0.3226\ldots$ is an infinite product over the prime numbers.
\end{thm}

It should be noted that the constant $C$ is studied elsewhere. Specifically, it is the asymptotic density of integers $n$ such that $n$ and $n+1$ are both square-free \cite{Mirsky1947} and is approximately equal to 0.3226. Moreover, $(1+C)/2$ is known as the Feller--Tornier constant and is equal to the asymptotic density of numbers with an even number of non-unitary prime divisors (see Feller--Tornier \cite{FellerTornier}).

It should also be noted that some tuples $(n, j, k)$ will result in an $I$-graph that is not connected. Our next result enumerates the density of tuples that result in a connected $I$-graph. Specifically, we show that about 98.3\% of all choices for $(n,j,k)$ result a connected graph.

\begin{thm} \label{igraphsconnected}
Let $A(N)$ count the number of tuples $(n, j, k)$ with $3 \leq n \leq N$, $k \leq n/2$ and $1 \leq j \leq k$. Let $C(N)$ count the subset of these tuples such that $I(n, j, k)$ is connected. Then 
$$\lim_{N \rightarrow \infty} \frac{C(N)}{A(N)} = \frac{1}{\zeta(6)} = \frac{945}{\pi^6} = 0.98295\ldots.$$
\end{thm}

The above results are, in a strong sense, number-theoretic, in that they arise from the counting of integer tuples satisfying various properties. Indeed, it may be considered more natural to perform the counting amongst the isomorphism classes of the graphs themselves, rather than the tuples. Our second set of results acts upon this consideration, starting with the following, which provides an asymptotic estimate for the number of isomorphism classes of the $I$-graphs with $n \leq N$. We first remind the reader that we write $f(N) \sim g(N)$ to mean that $f(N)/g(N) \rightarrow 1$ as $N \rightarrow \infty$.

\begin{thm} \label{countingisomorphisms}
    Let $CI(N)$ denote the count of isomorphism classes of $I$-graphs $I(n, j, k)$ with $n \leq N$. Then
    $$CI(N) \sim \frac{5}{16} N^2$$
    as $N \rightarrow \infty$.
\end{thm}

The above theorem immediately tells us that, on average, there are greater than $c N$ graphs in each isomorphism class. To see this directly, one notes that when we count all the tuples $(n, j, k)$ with $3 \leq n \leq N$, $k \leq n/2$ and $1 \leq j \leq k$, we furnish a result asymptotic to a constant multiple of $N^3$. 

The remaining two theorems provide us with the density of our interesting subclasses amongst the isomorphism classes.

\begin{thm} \label{igraphsextendgpgsisomorphism}
    Let $CP(N)$ denote the count of isomorphism classes of Peterson graphs $P(n,k)$ with $n\leq N$ and let $CI(N)$ denote the count of isomorphism classes of $I$-graphs with $n \leq N$. Then
    $$\lim_{N\to\infty}\frac{CP(N)}{CI(N)}=\frac{4(\pi^2-3)}{5 \pi^2} = 0.55683\dots.$$
\end{thm}

This says that approximately 55.7\% of isomorphism classes of $I$-graphs are isomorphism classes of generalised Peterson graphs.

\begin{thm} \label{igraphsconnectedisomorphism}
    Let $CI_c(N)$ denote the count of isomorphism classes of connected $I$-graphs with $n\leq N$ and let $CI(N)$ denote the count of isomorphism classes of $I$-graphs. Then
    $$\lim_{N\to\infty}\frac{CI_c(N)}{CI(N)}=\frac1{\zeta(2)}=0.60793\dots.$$
\end{thm}

This says that approximately 60.8\% of isomorphism classes of $I$-graphs are isomorphism classes of connected $I$-graphs. It is also interesting to note that this is the asymptotic density of square-free integers or, alternatively, the probability that one picks a pair of coprime integers. These two results illustrate the importance of counting isomorphism classes of the graphs, rather than simply counting all permissible tuples. 

One may combine the two above results to see that
$$\lim_{N \rightarrow \infty} \frac{CP(N)}{CI_c(N)} = \frac{2 (\pi^2-3)}{15} = 0.91594\ldots.$$
In short, the overwelming majority of connected $I$-graphs are, in fact, GPGs, and there are many graphs in each isomorphism class. However, of the $I$-graphs that are not connected, there are many different isomorphism classes with relatively few graphs in each class.

\section{Preliminary Lemmas and Results}

\subsection{Estimates for Sums} \label{estimates}

To prove Theorems \ref{igraphsextendgpgs} and \ref{igraphsconnected}, we will state and prove a couple of useful lemmas. Some of these sums will be of specific interest to number theorists.

\begin{lem}\label{countingcoprimenumbers}
Let $n \in \mathbb{N}$. We have that
\begin{equation}
\sum_{\substack{k \leq m \\ (k, n) = 1}} 1 = \frac{m \varphi(n)}{n} + O( 2^{\omega(n)}).
\end{equation}
\end{lem}

\begin{proof}
We proceed directly through
\begin{eqnarray*}
\sum_{\substack{k \leq m \\ (k, n)=1}} 1 & = & \sum_{k \leq m} \sum_{d | (k, n)} \mu(d) \\
& = & \sum_{d | n} \mu(d) \sum_{\substack{k \leq m \\ d | k} } 1 \\
& = & \sum_{d | n} \mu(d) \bigg( \frac{m}{d} + O(1) \bigg) \\
& = & m \sum_{d |n} \frac{\mu(d)}{d} + O\bigg(  \sum_{d|n} | \mu(d) | \bigg) \\
& = & \frac{m \varphi(n)}{n} + O( 2^{\omega(n)}).
\end{eqnarray*}
\end{proof}

\begin{lem}\label{summingcoprimenumbers}
Let $n \in \mathbb{N}$. We have that
\begin{equation}
\sum_{\substack{k \leq m \\ (k, n) = 1}} k = \frac{m^2 \varphi(n)}{2 n} + O(m 2^{\omega(n)}).
\end{equation}
\end{lem}

\begin{proof}
We proceed directly through
\begin{eqnarray*}
\sum_{\substack{k \leq m \\ (k, n)=1}} & = & \sum_{k \leq m} k \sum_{d | (k, n)} \mu(d) \\
& = & \sum_{d | n} \mu(d) \sum_{\substack{k \leq m \\ d | k} }k \\
& = & \sum_{d | n} \mu(d) \bigg( \frac{m^2}{2d} + O(m) \bigg) \\
& = & \frac{m^2}{2} \sum_{d |n} \frac{\mu(d)}{d} + O\bigg( m \sum_{d|n} 1 \bigg) \\
& = & \frac{m^2 \varphi(n)}{2n} + O(m 2^{\omega(n)}).
\end{eqnarray*}

\end{proof}

\begin{lem}\label{phisquared}
We have that
\begin{equation}
\sum_{n \leq N} \varphi^2(n) = C \frac{N^3}{3} + O(N^2 \log N)
\end{equation}
where
$$C = \prod_{p} \bigg( 1-\frac{2}{p^2}\bigg),$$
the product being over all prime numbers $p$.
\end{lem}
\begin{proof}
Set $k=0$ in Equation 30 of \cite{Mirsky}.
\end{proof}

\begin{lem}\label{nphi}
We have that
\begin{equation}
\sum_{n \leq N} n \varphi(n) = \frac{2}{\pi^2} N^3 +o(N^3) .
\end{equation}
\end{lem}
\begin{proof}
This follows from the formula
$$\sum_{n \leq N} \varphi(n) = \frac{3}{\pi^2} N^2 + o(N^2)$$
and the method of partial summation. One can see Bateman--Diamond \cite{BatemanDiamond} for more details.
\end{proof}

\subsection{Results on Counting $I$-Graphs}\label{counting_sec}

For the proofs of Theorems \ref{countingisomorphisms}, \ref{igraphsextendgpgsisomorphism} and \ref{igraphsconnectedisomorphism}, we require exact formulas for the number of isomorphism classes of $I$-graphs and GPGs of order $2n$. These are provided by Petkov\v{s}ek and Zakraj\v{s}ek \cite{PetkovsekZakrajsek} and we state these below.

\begin{lem}\label{I_count}
    Let $n=p_1^{k_1}p_2^{k_2}\dots p_{\omega(n)}^{k_{\omega(n)}}$ be the prime factorization of $n$. Then the number $I(n)$ of isomorphism classes of $I$-graphs on $2n$ vertices is given by
    \begin{align}
        I(n)=\frac1{4}\sum_{i=1}^4\prod_{j=1}^{\omega(n)}g_i\left(p_j^{k_j}\right)-\begin{cases}
            2\tau(n)-1,&n\text{ even},\\
            \tau(n),&n\text{ odd},
        \end{cases}
    \end{align}
    where
    \begin{align}
        g_1(p^k)&=\frac{(p+1)p^k-2}{p-1},\\
        g_2(p^k)&=\begin{cases}
            4k,&p=2,\\
            2k+1,&p<2,
        \end{cases}\\
        g_3(p^k)&=\begin{cases}
            2,&p=2\text{ and }k=1,\\
            4(k-1),&p=2\text{ and }k\geq2,\\
            2k+1,&p>2,
        \end{cases}\\
        g_4(p^k)&=\begin{cases}
            2,&p=2,\\
            2k+1,&p\equiv1\mod4,\\
            1,&p\equiv3\mod4.
        \end{cases}
    \end{align}
\end{lem}

\begin{lem}\label{connectedI_count}
    The number $I_c(n)$ of isomorphism classes of connected $I$-graphs on $2n$ vertices is given by
    \begin{align}
        I_c(n)=\frac1{4}\left(\frac{J_2(n)}{\varphi(n)}+r(n)+s(n)+t(n)\right)-\begin{cases}
            1, & n\text{ odd}\\
            2, & n\equiv0\mod{4}\\
            3, &n\equiv2\mod{4}
        \end{cases},
    \end{align}
    where
    \begin{align}
        t(n)=\begin{cases}
            2^{\omega(n)}+2^{\omega(n/2)},&n\text{ even},\\
            2^{\omega(w)},&n\text{ odd}.
        \end{cases}
    \end{align}
\end{lem}

\begin{lem}\label{peterson_count}
    The number $P(n)$ of isomorphism classes of generalised Peterson graphs on $2n$ vertices is given by
    \begin{align}P(n)=\frac1{4}(2n-\varphi(n)-2\gcd(n,2)+r(n)+s(n)).\end{align}
\end{lem}

\section{Proofs of Main Theorems }

\subsection{Proof of Theorem \ref{igraphsextendgpgs}}

To prove Theorem \ref{igraphsextendgpgs}, we require the following result of Boben, Pisanski and \v{Z}itnik \cite{Boben}.

\begin{thm}
    A graph $I(n, j, k)$ is a GPG if and only if $(n,j)=1$ or $(n,k)=1$.
\end{thm}

We may now proceed directly. As in the theorem statement, we let $A(N)$ count the number of tuples $(n, j, k)$ with $3 \leq n \leq N$, $k \leq n/2$ and $1 \leq j \leq k$. It follows then that
\begin{eqnarray*}
A(N) & = & \sum_{3 \leq n \leq N} \sum_{k \leq \lfloor n/2 \rfloor} \sum_{j \leq k} 1 \\
& = & \frac{1}{2} \sum_{3 \leq n \leq N} \lfloor n/2 \rfloor (\lfloor n/2 \rfloor + 1) \\
& = & \frac{N^3}{24} + O(N^2).
\end{eqnarray*}

Let $B(N)$ count the number of these tuples such that $I(n, j, k)$ is a generalised Petersen graph. Equivalently, $B(N)$ is the count of triples $(n, j, k)$ such that $3 \leq n \leq N$, $1 \leq j \leq k$, $1 \leq k \leq n/2$ and $(n, j) = 1$ or $(n,k)=1$. Considering separately the cases where $(n,k) = 1$ and where $(n,k) > 1$ we have that
$$B(N) = \sum_{3 \leq n \leq N} \sum_{\substack{k \leq n/2 \\ (k, n) = 1}} \sum_{j \leq k} 1 + \sum_{3 \leq n \leq N}  \sum_{\substack{k \leq n/2 \\ (k, n) > 1}} \sum_{\substack{j < k \\ (n,j) = 1}} 1.$$

To simplify our working, we write
$$B_1(N) = \sum_{\substack{k \leq n/2 \\ (k, n) = 1}} \sum_{j \leq k} 1$$
and
$$B_2(N) = \sum_{\substack{k \leq n/2 \\ (k, n) > 1}} \sum_{\substack{j < k \\ (n,j) = 1}} 1$$
and estimate each of these in turn using the tools developed in Section \ref{estimates}. From Lemma \ref{summingcoprimenumbers}, it follows directly that
$$B_1(N) = \frac{1}{8} n \varphi(n) + O(n 2^{\omega(n)}).$$

We can apply Lemma \ref{countingcoprimenumbers} to the second sum to get
\begin{eqnarray*}
B_2(N) & = & \frac{\varphi(n)}{n} \sum_{\substack{k \leq n/2 \\ (k,n)>1}} k + O(n 2^{\omega(n)}) \\
& = & \frac{\varphi(n)}{n} \bigg(\sum_{k \leq n/2} k - \sum_{\substack{k \leq n/2 \\ (k,n)=1}} k \bigg) + O(n 2^{\omega(n)}).
\end{eqnarray*}
An application of Lemma \ref{summingcoprimenumbers} gives us that
$$B_2(N)  =  \frac{1}{8} n \varphi(n) - \frac{1}{8} \varphi^2(n) + O(n 2^{\omega(n)}).$$
We can feed these results back into our expression for $B(N)$ to get
\begin{eqnarray*}
B(N) & = & \sum_{3 \leq n \leq N} \bigg( \frac{1}{4} n \varphi(n) - \frac{1}{8} \varphi^2(n)\bigg) + O\bigg(\sum_{n\leq N} n 2^{\omega(n)} \bigg) \\
& = & \frac{1}{4} \sum_{n \leq N} n \varphi(n) - \frac{1}{8} \sum_{n \leq N} \varphi^2(n) + O\bigg(\sum_{n\leq N} n 2^{\omega(n)} \bigg)
\end{eqnarray*}
A direct application of Lemmas \ref{phisquared} and \ref{nphi} gives us that 
$$B(N) = \bigg(\frac{1}{2 \pi^2} - \frac{C}{24} \bigg) N^3 + O\bigg(\sum_{n\leq N} n 2^{\omega(n)} \bigg)$$
where $C$ is specified in Lemma \ref{phisquared}. We now note that $2^{\omega(n)} \leq d(n)$ where $d(n)$ denotes the number of divisors and use the well-known estimate
$$\sum_{n \leq N} d(n) = N \log N + O(N).$$
This completes the proof of Theorem \ref{igraphsextendgpgs}.

\subsection{Proof of Theorem \ref{igraphsconnected}}

The proof relies on Theorem 8 of Oliveira--Vinagre \cite{OliveiraVinagre} which we state as follows.

\begin{thm}\label{connected}
    The graph $I(n,j,k)$ is connected if and only if $(n,j,k)=1$.
\end{thm}

Let $C(N)$ count the number of tuples $(n, j, k)$ with $3 \leq n \leq N$, $k \leq n/2$ and $1 \leq j \leq k$, as well as the added condition that $(n,j,k)=1$. This is the sum
$$C(N) = \sum_{3 \leq n \leq N} \sum_{k \leq  n/2} \sum_{\substack{j \leq k \\ (n, j, k) = 1}} 1$$

It should be noted that we have already counted a lot of these simply by virture of counting the tuples such that $(n,j)=1$ or $(n,k) = 1$ in the previous section. Indeed, all that is left is to count all of the tuples such that $(j,k)=1$, $(n,k)>1$ and $(n,j)>1$, and add this to $B(N)$ to get $C(N)$. 

However, we compute the sum $C(N)$ directly using an idea of Toth (see Remark 1 on P13 of \cite{toth}) which is altogether not too dissimilar to the ideas used in the proofs of Lemma \ref{countingcoprimenumbers} and \ref{summingcoprimenumbers}. Proceeding directly, we write
$$C(N) = \sum_{3 \leq n \leq N} \sum_{k \leq  n/2} \sum_{j \leq k} \sum_{d | (n,j,k)} \mu(d).$$
We switch the order of summation to get
\begin{equation} \label{CSummationSwitched}
    C(N) = \sum_{d=1}^N \mu(d) \sum_{1 \leq n_1\leq N/d} \sum_{k_1 \leq  n/2d} \sum_{j_1 \leq k/d} 1
\end{equation}
where we have written $n=dn_1$, $k = dk_1$ and $j = d j_1$. The evaluation of the inner triple sum is as before; we get that
\begin{eqnarray*}
    \sum_{1 \leq n_1\leq N/d} \sum_{k_1 \leq  n/2d} \sum_{j_1 \leq k/d} 1 & = & \sum_{1 \leq n_1\leq N/d} \sum_{k_1 \leq  n/2d} \bigg( \frac{k_1}{d} + O(1) \bigg) \\
    & = & \sum_{1 \leq n_1 \leq N/d} \bigg(\frac{1}{8} \bigg(\frac{n_1}{d}\bigg)^2 + O \bigg(\frac{n_1}{d} \bigg) \bigg) \\
    & = & \frac{1}{24} \frac{N^3}{d^6} + O\bigg(\frac{N^2}{d^3}\bigg).
\end{eqnarray*}
Substituting this directly into Equation \ref{CSummationSwitched} gives us that
$$C(N) = \frac{N^3}{24} \sum_{d=1}^N \frac{\mu(d)}{d^6} + O(N^2).$$
We note that the sum may be evaluated
\begin{eqnarray*}
    \sum_{d=1}^N \frac{\mu(d)}{d^6} & = & \sum_{d=1}^{\infty} \frac{\mu(d)}{d^6} + O \bigg( \sum_{d > N} \frac{1}{d^6} \bigg) \\
   & = & \frac{1}{\zeta(6)} + O(N^{-5}) \\
\end{eqnarray*}
and this completes the proof.

\subsection{Proof of Theorem \ref{countingisomorphisms}}

The purpose of this section is to evaluate asymptotically the sum
$$CI(N) = \sum_{n \leq N} I(n)$$
where $I(n)$ is as provided in Lemma \ref{I_count}. Working directly, it follows that
$$CI(N) = \frac{1}{4} \sum_{n \leq N} \sum_{i=1}^4 \prod_{j=1}^{\omega(n)} g_i(p_j^{k_j}) +O \bigg( \sum_{n \leq N} \tau(n) \bigg).$$
To ease the notation, we write 
$$g_i(n) = \prod_{j=1}^{\omega(n)} g_i(p_j^{k_j})$$
where $n = p_1^{k_1} \cdots p_{\omega(n)}^{k_{\omega(n)}}$. It then follows that
\begin{equation} \label{estimateInPieces}
    CI(N) = \frac{1}{4} ( \Sigma_1 + \Sigma_2 + \Sigma_3 + \Sigma_4) +O \bigg( \sum_{n \leq N} \tau(n) \bigg)
\end{equation}
where 
$$\Sigma_i = \sum_{n \leq N} g_i(n).$$
The reader can glance at the structures of the piecewise functions $g_2$, $g_3$ and $g_4$ to realise that it is expected that sums of these functions should all be of similiar asymptotic order. Indeed, each piece, and therefore each of these three functions can all be be bounded above by $g_u(p^k) = (k+1)^2$. This may seem, at a glance, wasteful, but as we will see this function lends itself swiftly to estimation using analytic techniques.

As such, we have that
$$\Sigma_2 + \Sigma_3 + \Sigma_4 = O\bigg( \sum_{n \leq N} g_u(n) \bigg)$$
where the effect of $g_u$ on some integer $n$ is in the normal multiplicative manner to each of its prime factors. We will now proceed to estimate this sum, after which we will estimate what turns out to be the main term, namely $\Sigma_1$.

Our method of estimation will be similiar for both; namely, we construct a Dirichlet series that embeds the function of interest into its coefficients. Then, we apply the Ikehara--Wiener theorem to extract an asymptotic formula for the partial sums.

Consider the Euler product
$$F_u(s) = \prod_{p} \bigg(1 + \frac{4}{p^{s}} + \frac{9}{p^{2s}} + \frac{16}{p^{3s}} + \cdots \bigg)$$
which converges for $\text{Re}(s) > 1$ (see \cite{BatemanDiamond} for further theory). In the region of convergence, we have that
$$F_u(s) = \sum_{n=1}^{\infty} \frac{g_u(n)}{n^s}.$$
We now wish to appeal to analytic properties of the above Dirichlet series in order to estimate the sum $\sum_{n \leq N} g_u(n)$ \textit{viz.} the Ikehara--Wiener theorem. This theorem comes in many forms; a useful one is provided from Theorem 2.4.1 in Cojocaru and Murty \cite{cojocarumurty}.

\begin{thm} \label{ikeharawiener1}
Let
$$F(s) = \sum_{n=1}^{\infty} \frac{a_n}{n^s}$$
be a Dirichlet series with non-negative coefficients converging for $\text{Re}(s) >1$. Suppose that $F(s)$ extends analytically at all points on $\text{Re}(s)=1$ apart from $s=1$, and that at $s=1$ we can write
$$F(s) = \frac{H(s)}{(s-1)^{1-\alpha}}$$
for some $\alpha \in \mathbb{R}$ and some $H(s)$ holomorphic in the region $\text{Re}(s) \geq 1$ and non-zero there. Then
$$\sum_{n \leq x} a_n \sim \frac{c x}{( \log x)^{\alpha}}$$
with
$$c:=\frac{H(1)}{\Gamma(1 - \alpha)}$$
where $\Gamma$ is the Gamma function.
\end{thm}

It is fairly easy to see that $F_u(s)$ satisfies the conditions of the above theorem. Actually, it follows from the derivation of (1.2.10) of Titchmarsh \cite{titchmarsh1986theory} that 
$$F_u(s) = \frac{\zeta^4(s)}{\zeta(2s)}$$
and we can now borrow from the known properties of the Riemann zeta-function. Specifically, we know that at $s=1$ we may write 
$$F_u(s) = \frac{H(s)}{(s-1)^4}$$
where $H(s)$ is holomorphic and non-zero in the region $\text{Re}(s) \geq 1$. It follows that
$$\sum_{n \leq N} g_u(n) = O(N \log^3 N).$$
Substituting this back into Equation \ref{estimateInPieces} along with the fact that $\sum_{n \leq N} \tau(n) = O(N \log N)$ (see \cite{BatemanDiamond}) we have that
\begin{equation}
CI(N) = \frac{1}{4}  \Sigma_1 + O(N \log^3 N).
\end{equation}
Finally, we seek to estimate $\Sigma_1$. An upper bound will not do this time. It is both fortunate and intriguing that the function $g_1$ naturally resolves via the Riemann zeta-function.

\begin{lem}
    We have for $\text{Re}(s) > 2$ that
    $$\sum_{n = 1}^{\infty} \frac{g_1(n)}{n^s} = \frac{\zeta(s)^2 \zeta(s-1)}{\zeta(2s)}.$$
\end{lem}

\begin{proof}
We start by expressing the Dirichlet series for the multiplicative function $g_1$ as an Euler product.
    \begin{align*}
        \sum_{n=1}^\infty \frac{g_1(n)}{n^s}
        &=\prod_p\left(1+\frac{g_1(p)}{p^s}+\frac{g_1(p^2)}{p^{2s}}+\frac{g_1(p^3)}{p^{3s}}+\dots\right) \\
        &=\prod_p\left(1+\frac{p+1}{p-1}\left(1+\frac1{p^{s-1}}+\frac1{p^{2(s-1)}}+\dots\right)-\frac{2}{p-1}\left(\frac1{p^s}+\frac1{p^{2s}}+\dots\right)\right).
    \end{align*}
    Now, $\sum_{n=1}^\infty 1/p^{ns}=1/(p^s-1)$ and similarly $\sum_{n=1}^\infty 1/p^{n(s-1)}=1/(p^{s-1}-1)$ so \begin{align*}
        \sum_{n=1}^\infty\frac{g_1(n)}{n^s}&=\prod_p\left(1+\frac{p+1}{p-1}\cdot\frac1{p^{s-1}-1}-\frac2{p-1}\cdot\frac1{p^s-1}\right)\\
        &=\prod_p\left(\frac{p^s+1}{p^s-1}\cdot\frac1{1-p^{1-s}}\right) \\
        &=\prod_p\left(\frac{p^s+1}{p^s-1}\right)\cdot\prod_p\left(\frac1{1-p^{1-s}}\right)=\frac{\zeta(s)^2\zeta(s-1)}{\zeta(2s)}.
    \end{align*}
    The final equality follows directly from the proof of (1.2.8) of Titchmarsh \cite{titchmarsh1986theory}.
\end{proof}

It is interesting to note, from Lemma 3.13 of \cite{PetkovsekZakrajsek}, that before the direct evaluation one has
$$g_1(n) = \frac{1}{\varphi(n)} \sum_{a \in \mathbb{Z}_n^*} \gcd(n, a-1)^2.$$
It is striking that the above sum should find itself arising quite naturally in terms of the Riemann zeta-function. As before, we will use the analytic properties of this Dirichlet series to estimate $\sum_{n\leq N}g_1(n)$ \textit{viz.} the Ikehara--Weiner theorem. The following useful form of this theorem can be found in Murty \cite{Murty}.

\begin{thm}\label{ikeharawiener2}
    Suppose $F(s)=\sum_{n=1}^\infty a_n/n^s$ is a Dirichlet series with non-negative coefficients that is convergent for Re$(s)>c>0$. If $F(s)$ extends to a meromorphic function in the region Re$(s)\geq c$ with only a simple pole at $s=c$ and residue $R$ then $$\sum_{n\leq x}b_n\sim\frac{Rx^c}{c}.$$
\end{thm}

Here we have $F(s)=\sum_{n=1}^\infty g_1(n)/n^s$. Since we have just evaluated this to be $\zeta(s)^2\zeta(s-1)/\zeta(2s)$, it is clear $F(s)$ extends to a meromorphic function in the region $Re(s)\geq 2$ with a simple pole at $s=2$. Clearly, we have that
$$\quad\text{res}_{s\to2}\,\zeta(s-1)=1$$
and so it follows that
 $$\sum_{n\leq N}g_1(n)\sim\frac{\zeta^2(2)}{2 \zeta(4)}N^2 = \frac{5}{4} N^2.$$ Substituting this back into the expression for $CI(N)$ gives
$$CI(N)=\frac{5}{16}N^2+O(N\log^3N)\sim\frac{5}{16}N^2$$
and this completes the proof.\hfill$\Box$

\subsection{Proof of Theorem \ref{igraphsextendgpgsisomorphism}}
    Let $CP(N)$ denote the count of isomorphism classes of Peterson graphs $P(n,k)$. Using Lemma \ref{peterson_count} we have that
    \begin{align*}
        CP(N)&=\sum_{n\leq N}P(n)\\
        &=\sum_{n\leq N}\frac1{4}(2n-\varphi(n)-2\gcd(n,2)+r(n)+s(n)) \\
        &=\frac1{4}\left(\sum_{n\leq N}2n-\sum_{n\leq N}\varphi(n)-\sum_{n\leq N}2\gcd(n,2)+r(n)+s(n)\right).
    \end{align*}
    It can now be seen that each of these summations can be evaluated using classic methods. First, as seen in \cite{apostol}, we have that
        $$\sum_{n \leq N} \varphi(n) = \frac{3}{\pi^2} N^2 + O(N\log N),$$
    and
        $$\sum_{n\leq N}2n=N^2+O(N).$$
    Finally, it is clear that $\gcd(n,2)=O(1)$ and both
    $$s(n)=\begin{cases}
        0,&4|n\mbox{ and }\exists p : (p|n\mbox{ and }p\equiv3\mod 4) \\
        2^{\psi(n)},&\mbox{otherwise}
    \end{cases}$$
    and
    $$r(n)=\begin{cases}
        2^{\omega(n)},&n\equiv1\mod2\mbox{ and }n\equiv4\mod8,\\
        2^{\omega(n)-1},&n\equiv2\mod4,\\
        2^{\omega(n)+1},&n\equiv0\mod8.
    \end{cases}$$
    are $O(2^{\omega(n)})$. Since it is known that
        $$2^{\omega(n)}=\sum_{d|n}|\mu(d)|,$$
    then it follows by Abel Summation that
    $$\sum_{n\leq N}2^{\omega(n)}=O(N\log N).$$
    Combining these results we obtain
    \begin{align*}
        CP(N)&=\frac1{4}\left(N^2+N-\frac3{\pi^2}N^2\right)+O(N\log N)\\
        &=\frac{\pi^2-3}{4\pi^2}N^2+O(N\log N).
    \end{align*}
    Now, as proven in Theorem \ref{countingisomorphisms}, we have that
    $$CI(N)=(5/16) N^2+O(N\log^3N)$$
    and so we are now able to compute the final result, for
    \begin{align*}
        \frac{CP(N)}{CI(N)}&=\frac{\frac{\pi^2-3}{4\pi^2}N^2+O(N\log N)}{\frac5{16}N^2+O(N\log^3N)}
    \end{align*}
        and so
        
    \begin{align*}
        \lim_{N\to\infty}\frac{CP(N)}{CI(N)}&=\frac{4(\pi^2-3)}{5\pi^2}
    \end{align*}
    This completes the proof of Theorem \ref{igraphsextendgpgsisomorphism}.\hfill$\Box$


\subsection{Proof of Theorem \ref{igraphsconnectedisomorphism}}
    Let $CI_c(N)$ denote the count of isomorphism classes of connected $I$-graphs $I(n,j,k)$. Using Lemma \ref{connectedI_count} we have
    \begin{align*}
        CI_c(N)&=\sum_{n\leq N}I_c(n)\\
        &=\sum_{n\leq N}\frac1{4}\left(\frac{J_2(n)}{\varphi(n)}+r(n)+s(n)+t(n)\right)-\begin{cases}
            1, & n\text{ odd}\\
            2, & n\equiv0\mod{4}\\
            3, &n\equiv2\mod{4}.
        \end{cases}
    \end{align*}
    We first note that the piecewise term will simply be $O(1)$ and since \begin{align}
        t(n)=\begin{cases}
            2^{\omega(n)}+2^{\omega(n/2)},&n\text{ even},\\
            2^{\omega(w)},&n\text{ odd},
        \end{cases}
    \end{align}
    we have that $r(n)$, $s(n)$ and $t(n)$ are all $O(N\log N)$. The main focus of this sum will therefore be $J_2(n)/\varphi(n)$. This is simply the Dedekind psi-function and it is known that
        $$\sum_{n\leq N}\frac{J_2(n)}{\varphi(n)}=\frac{15}{2\pi^2}N^2+O(N\log N).$$
    A sketch of this proof can be found in Chapter 3 of \cite{apostol}. Combining these results we obtain the following
    $$CI_c(N)=\frac{15}{8\pi^2}N^2+O(N\log N).$$
    Now, as proven in Theorem \ref{countingisomorphisms} $CI(N)=5/16\cdot N^2+O(N\log^3N)$. We are now able to compute the final result.
    \begin{align*}
        \frac{CI_c(N)}{CI(N)}&=\frac{\frac{15}{8\pi^2}N^2+O(N\log N)}{\frac5{16}N^2+O(N\log^3N)}\\
        \lim_{N\to\infty}\frac{CI_c(N)}{CI(N)}&=\frac6{\pi^2}=\frac1{\zeta(2)}.
    \end{align*}
    This completes the proof of Theorem \ref{igraphsextendgpgsisomorphism}.\hfill$\Box$
    
\section*{Acknowledgements}

The authors would like to thank Randell Heyman for continuously pointing them in the right direction on the evaluation of many of the sums involved in Theorem \ref{igraphsextendgpgs} and \ref{igraphsconnected}.

\clearpage

\bibliographystyle{plain}

\bibliography{biblio}

\end{document}